%% file: Euclidean-is-ADC-short.tex
\documentclass[11pt,letterpaper]{article}

\pdfoutput=1

\usepackage[bookmarks=false]{hyperref}

\input{extended-preamble.tex}

\addtotextwidth{2cm}
\addtotextheight{2.5cm}

\input{preamble.tex}

\newcommand{\bilinform}[2]{\left\langle#1,#2\right\rangle}

\begin{document}

\title{\vspace{-7ex}
	\LARGE\bfseries Euclidean quadratic forms are ADC forms:\\
			a short proof\\[2ex]}

\author{\large\bfseries France Dacar, Jo\v{z}ef Stefan Institute\\
	{\large\ttfamily France.Dacar@ijs.si}}

\date{\large\bfseries August 23, 2012\\
	Last edited \today\\[9.5ex]}

\hypersetup{
  pdftitle={Euclidean quadratic forms are ADC forms: a short proof},
  pdfauthor={France Dacar},
}

\maketitle

\begin{abstract}
This note presents a short, transparent proof
of the theorem that every Euclidean quadratic form over a~normed integral domain
	is an Aubry-Davenport-Cassels form.
The theorem, as formulated in the note,
allows besides quadratic terms also linear and constant terms,
imposes no restrictions on the characteristic of the integral domain,
and makes no unnecessary assumptions about the norm.
\end{abstract}

\section{\protect\rule{0pt}{4ex}Introduction}

In this note we state and prove a~sharpened version of Theorem~8 in~\cite{PeteLClark},
which says that every Euclidean quadratic form
over a~certain kind of a~normed integral domain is an ADC form;
this theorem is a generalization of the Davenport-Cassels lemma.
The sharpened theorem presented in this note allows for linear and constant terms
(so~that there is a~quadratic polynomial instead of a~quadratic form,
	albeit with a~Euclidean homogeneous quadratic~part),
imposes no restrictions on the characteristic of the integral domain,
and makes no unnecessary assumptions about the norm.
The theorem has an impressive provenance:
its central idea was generalized and clarified, in stages,
from Aubry through Cassels, Davenport, Weil, and Deligne to Serre.

\inskip

\section{Definitions}

Let $R$ be a commutative ring with unity $1\neq 0$.

A~\notion{discrete multiplicative norm on~$R$} (shorter, a~\notion{norm on~$R$})
is a mapping $\norm{\anon}\colon R\to\NN$
that satisfies the following two conditions:
\begin{items}
\item[(N0)\:] For every $x\in R$, $\norm{x}=0$ \iff\ $x=0$.
\item[(N1)\:] For all $x,\,y\in R$, $\norm{x y}=\norm{x}\norm{y}$.
\end{items}
\pagebreak[3]
Since $\norm{1}\norm{1}=\norm{1\cdot1}=\norm{1}\neq 0$,
we~have~$\norm{1}=1$, thus $\norm{\anon}$ is a~homomorphism of multiplicative monoids,
and as such it maps every unit of~$R$
to the only invertible element $1$ of the multiplicative monoid~$\NN$.

Let $\norm{\anon}$ be a~discrete multiplicative norm on~$R$.

If $x$ and $y$ are any non-zero elements of $R$,
then $\norm{x y}=\norm{x}\norm{y}\neq 0$, thus $x y\neq 0$;
the ring $R$ is an integral domain.
Let $K$ be the field of fractions of~$R$.
The given norm $\norm{\anon}$ on~$R$
extends in a~unique way to a~mapping $\norm{\anon}\colon K\to\QQ_{\geq0}$
satisfying the condition~(N1) $\bigl($which then also satisfies the condition~(N0)$\bigr)$:
if $x=a/b$ with $a, b\in R$ and $b\neq 0$, then $\norm{x}=\norm{a}/\norm{b}$.
The extended mapping is a~\notion{multiplicative norm} on~$K$.

\bigskip

A~\notion{form over a ring $R$} is a~homogeneous polynomial in $R[X_1,\dots,X_d]$ ($d>0$),
where $X_i$ are formal variables;
we shall write $X=\tuple{X_1,\dots,X_d}$~etc.

Let $q\in R[X_1,\dots,X_d]$ be a~quadratic form.
The polynomial $\bilinform{X}{Y}_{\!q}\defeq q(X+Y)-q(X)-q(Y)\in R[X_1,\dots,X_d,Y_1,\dots,Y_d]$
is a~bilinear form said to be \notion{associated with~$q$}.
If $T$ is a~formal variable different from the formal variables $X_i$ and $Y_i$, then
\begin{equation*}
q(X+TY) \Eq q(X) + \bilinform{X}{Y}_{\!q}T + q(Y)\,T^2~;
\end{equation*}
it is the fact that the coefficient at $T^2$ is $q(Y)$
which we will find useful in the proof of the theorem below.

Let $R$ be an integral domain with the field of fractions $K$.

A~form $g$ over $R$, in $d$ variables, is said to be an~\notion{ADC form}%
\footnote{ADC stands for Aubry-Davenport-Cassels.}
if for every $x\in K^d$ at which $g(x)\in R$ there exists $y\in R^d$ such that $g(y)=g(x)$.

Let $\norm{\anon}$ be a~discrete multiplicative norm on~$R$,
uniquely extended to a~multiplicative norm on $K$ (still written~$\norm{\anon}$).
A~form $g$ over $R$, of any degree,
is said to be \notion{Euclidean with respect to the norm~$\norm{\anon}$}
if for every $x\in K^d\setminus R^d$ there exists $y\in R^d$
such that $0<\norm{g(x-y)}<1$.

\inskip

\section{The theorem}

\inskip

\begin{theorem}\label{thm:generalized-Aubry-Davenport-Cassels-Weil-Deligne-Serre}
Let\/ $R$ be an integral domain, with the field of fractions\/~$K$,
and let\/ $\norm{\anon}$ be a~discrete multiplicative norm on\/~$R$,
extended to a~multiplicative norm\/~$\norm{\anon}$ on\/~$K$.
Let\/ $f=f_2+f_1+f_0\in R[X_1,\dots,X_d]$, where\/ $f_i$ is homogeneous of degree\/~$i$,
and\/ $f_2$ is Euclidean with respect to\/~$\norm{\anon}$.
If\/~$f$~has a~zero in\/~$K^d$, then it has a zero in\/~$R^d$.
\end{theorem}

\interskip

\begin{proof}
Let $x\in K^d$ be a~zero of~$f$;
if $x\in R^d$, we are done, so we assume that $x\notin R^d$.

We have $x=a/b$ for some $a\in R^d$ and $b\in R\setminus\set{0}$,
and there exists $y\in R^d$ such that $0 < \bignorm{f_2(x-y)} < 1$.
We have $x-y = v/b$ with $v = a-by\in R^d$.
For any $t\in K$ set $F(t)\defeq f(y + tv) = A t^2 + B t + C$,
where the coefficients $A=f_2(v)=f_2(x-y)\tinysp b^2\neq 0$, $C=f(y)$,
	and $B = f(y+v) - A - C$ are~in~$R$;
$\tau\defeq 1/b$ is a~zero of~$F$ because $x = y + v/b$.
Let~$\tau'$~be the other zero~of~$F$.
Since $\tau\tinysp\tau' = C/A$, we have $\tau' = C/\tau A = C/(A/b)$, where $A/b = -B - Cb$ is in~$R$;
since also $A/b = f_2(x-y)\,b$, it follows that $\norm{A/b} = \norm{f_2(x-y)}\norm{b} < \norm{b}$.
The~point $x'\defeq y + \tau'v$ is a~zero of~$f$,
and it can be represented as $x'=a'/b'$,
where $b' = A/b\in R\setminus\set{0}$,
$a' = b'y + Cv\in R^d$, and~$\norm{b'}<\norm{b}$.

If~the~zero~$x'$ of~$f$ is not yet in~$R^d$,
we repeat the procedure and construct another zero $x''=a''/b''$ of~$f$,
where $a''\in R^d$, $b''\in R\setminus\set{0}$, and $\norm{b''}<\norm{b'}$.
And so on.
The sequence $x$, $x'$, $x''$,~\dots\ of zeros of~$f$
eventually ends with a~zero $x^{(s)}\in R^d$~of~$f$.
\end{proof}

\thmskip

Looking at the constructed sequence $x = a/b$, $x' = a'/b'$, $x'' = a''/b''$, \dots\ of zeros of~$f$,
where $b' = f_2(x-y)\cdot b$, $b'' = f_2(x'-y')\cdot b'$, \dots\,,
we see how the Euclidean form $f_2$ forces termination
because the descending sequence of the norms $\norm{b}>\norm{b'}>\norm{b''}>\cdots$ must be finite.
Mark that the final fraction $a^{(s)}/b^{(s)}=x^{(s)}\in R$ may have $\norm{b^{(s)}}>1$.

\thmskip

\begin{corollary}
Let\/ $R$ be an integral domain with a~discrete multiplicative norm\/~$\norm{\anon}$.
If~a~quadratic form\/~$q$ over\/~$R$ is Euclidean with respect to\/~$\norm{\anon}$
then~it~is an ADC form.
\end{corollary}

\interskip

\begin{proof}
Given an arbitrary $r\in R$,
apply Theorem~\ref{thm:generalized-Aubry-Davenport-Cassels-Weil-Deligne-Serre} to $q-r$.
\end{proof}

\inskip

\section{A remark on a certain property of the norm}

Let $R$ be an integral domain with a~discrete multiplicative norm $\norm{\anon}$.
Suppose there exists a~Euclidean form $g$ over $R$,
in any number $d>0$ of variables and of any degree $m>0$.
Then the norm has the following property:
\begin{items}
\item[(N2)\:] If $x\in R$ has $\norm{x}=1$, then $x$ is a~unit of~$R$.
\end{items}
Indeed, let $a$ be a~non-zero non-unit of $R\tinysp$; we shall prove that $\norm{a}>1$.
The point $x\defeq\tuple{a^{-1},0,\dots,0}=a^{-1}e_1$ lies in $K^d\setminus R^d$,
therefore there exists $y\in R^d$ such that $0 < \norm{g(x-y)} < 1$,
that is, $0 < \norm{g(e_1-ay)} < \norm{a}^m$;
since $\norm{g(e_1-ay)}$ is an integer,
it follows that $\norm{a}^m>1$, whence $\norm{a}>1$.

Though in Theorem~\ref{thm:generalized-Aubry-Davenport-Cassels-Weil-Deligne-Serre}
we do not explicitly assume the property~(N2),
the presence of the~Euclidean quadratic form $f_2$ implies~it.
Theorem~8 in~\cite{PeteLClark} thus unnecessarily assumes the property (N2).%
\footnote{The proof of this theorem refers to the property~(N1),
which is a~typo: property (N2), as defined in~\cite{PeteLClark},
was property (N1) in an earlier draft version of the paper.}

\inskip


\end{document}

%% file: extended-preamble.tex
\input{general-preamble.tex}

%% file: general-preamble.tex
\usepackage{latexsym}
\usepackage{amssymb}
\usepackage{amsmath}
\usepackage{mathrsfs}  

\usepackage{amsthm}

\usepackage{graphicx}

\usepackage[margin=5mm,format=hang,justification=raggedright,font=small,labelfont=bf]{caption}

\usepackage{verbatim}

\usepackage{accents}

\usepackage{relsize}

\newcommand{\addtotextwidth}[1]{%
  \setlength{\hoffset}{.5\textwidth}
  \addtolength{\textwidth}{#1}
  \addtolength{\hoffset}{-.5\textwidth}}
\newcommand{\addtotextheight}[1]{%
  \setlength{\voffset}{.5\textheight}
  \addtolength{\textheight}{#1}
  \addtolength{\voffset}{-.5\textheight}}

\newlength{\itemsmargin}
\newlength{\itemslabelwidth}
\newlength{\itemslabelsep}
\newlength{\itemstopsep}
\newlength{\itemsitemsep}
\newenvironment{items}%
  {\begin{list}{}{\setlength{\topsep}{\itemstopsep}
                  \setlength{\leftmargin}{\itemsmargin}
                  \setlength{\labelwidth}{\itemslabelwidth}
		  \setlength{\labelsep}{\itemslabelsep}
		  \setlength{\itemindent}{0pt}
                  \setlength{\listparindent}{.75\parindent}
                  \setlength{\itemsep}{\itemsitemsep}
                  \setlength{\parsep}{0pt}
                  \setlength{\parskip}{0pt}}}
  {\end{list}}

\renewenvironment{itemize}[1][2em]%
  {\begin{list}{}%
    {\setlength{\topsep}{1ex}
     \setlength{\labelwidth}{#1}
     \setlength{\leftmargin}{\labelwidth}
     \addtolength{\leftmargin}{1.5em}
     \setlength{\labelsep}{.618em}
     \setlength{\itemindent}{0pt}
     \setlength{\listparindent}{.75\parindent}
     \setlength{\itemsep}{1ex}
     \setlength{\parsep}{0pt}
     \setlength{\parskip}{0pt}}}
  {\end{list}}
  
\newcommand{\notion}[1]{{\bfseries #1}}

\renewcommand{\iff}{if~and only~if}

\newcommand{\tinysp}{\mspace{1mu}}

\newcommand{\dtinysp}{\mspace{2mu}}

\newcommand{\NN}{\mathbb{N}}    

\newcommand{\QQ}{\mathbb{Q}}    

\newcommand{\Eq}{\:=\:}

\newcommand{\defeq}{\mathrel{\overset{\text{def}}{\tinysp=\!=\tinysp}}}

\renewcommand{\geq}{\geqslant}

\newcommand{\set}[1]{\{ #1 \}}                  

\newcommand{\nothing}{}

\newcommand{\setparenarr}{\let\larr(\let\rarr)\let\arrmarg\nothing}
\newcommand{\setanglearr}{\let\larr\langle\let\rarr\rangle\let\arrmarg\tinysp}

\setparenarr

\newcommand{\arr}[1]{\larr\arrmarg #1 \arrmarg\rarr}         
\newcommand{\tuple}{\arr}

\newcommand{\overlinepart}[2][1]%
  {{\setbox0=\hbox{$#2$}\accentset{\raisebox{0.0618ex}{\rule{#1\wd0}{0.09ex}}}{#2}}}

\newcommand{\underlinepart}[2][1]%
  {{\setbox0=\hbox{$#2$}\underaccent{\rule{0pt}{.18ex}\rule{#1\wd0}{0.0618ex}}{#2}}}

\newcommand{\bigtruthord}[1]%
  {\text{\larger$\boldsymbol{[}$}\dtinysp{#1}\dtinysp\text{\larger$\boldsymbol{]}$}}

\newcommand{\norm}[1]{\mathopen{\|}#1\mathclose{\|}}	
\newcommand{\bignorm}[1]{\bigl\|#1\bigr\|}

\newcommand{\anon}{\mathord{\halftinysp\rule[0.5ex]{0.5em}{0.5pt}\halftinysp}}

\newcommand{\farref}[1]{?.?}

\newcommand{\twohyphs}%
	{\mathord{\dtinysp\raisebox{.15ex}{\text{-}}\dtinysp\raisebox{.15ex}{\text{-}}\dtinysp}}

%% file: preamble.tex
\newtheoremstyle{myplain}
  {0pt}
  {0pt}
  {\itshape}
  {}
  {\bfseries}
  {.}
  {0.75em}
  {}

\newtheoremstyle{mydefinition}
  {0pt}
  {0pt}
  {}
  {}
  {\bfseries}
  {.}
  {0.75em}
  {}

\makeatletter
\renewcommand*\env@matrix[1][*\c@MaxMatrixCols c]{%
  \hskip -\arraycolsep
  \let\@ifnextchar\new@ifnextchar
  \array{#1}}
\makeatother

	{\begin{list}{}{%
		\setlength{\leftmargin}{\parindent}%
		\setlength{\rightmargin}{\parindent}}%
		\item[]\ignorespaces}%
	{\unskip\end{list}}

\theoremstyle{myplain}
\newtheorem{theorem}{Theorem} 

\newtheorem{corollary}[theorem]{Corollary}

\theoremstyle{mydefinition}

\newtheorem{remark}{Remark}

	{\begin{remark}}%
	{\hfill\qed\end{remark}}

\renewenvironment{proof}[1][Proof.]
	{\noindent{\bfseries #1}\hspace{.75em}\ignorespaces}
	{\hspace{\stretch{1}}\qed}
\newenvironment{proof*}[1][Proof.]
	{\noindent{\bfseries #1}\hspace{.75em}\ignorespaces}
	{\par}

\newcommand{\thmskip}{\bigskip\vspace{2ex}}

\newcommand{\interskip}{\bigskip}
\newcommand{\inskip}{\medskip}

\renewcommand{\emph}[1]{\textsl{#1\/}}

\renewcommand{\anon}{\mathord{\tinysp\rule[0.5ex]{0.5em}{0.5pt}\tinysp}}

\renewcommand{\defeq}{:=}

\newcommand{\pure}[1]%
	{\accentset{\raisebox{-0.25ex}[0pt][0pt]{$\smash{\scriptscriptstyle\rightharpoonup}$}}{#1}}

\newcommand{\bigtxtmtx}[1]%
	{\raisebox{-0.1ex}{\Large\boldmath$[$}#1\raisebox{-0.1ex}{\Large\boldmath$]$}}

\newcommand{\microrightharpoonup}{\includegraphics[hiresbb=true]{mp/micro-rightharpoonup-1.mps}}
\newcommand{\scriptpure}[1]%
	{\accentset{\raisebox{-0.25ex}[0pt][0pt]{\microrightharpoonup}}{#1}}